\documentclass[11pt]{article}
\usepackage{amssymb,amsmath,mathrsfs,amsthm,indentfirst} 
\usepackage[numbers,sort&compress]{natbib}
\begin{document}

\title{
\bf Analysis of a nonlinear necrotic tumor model with angiogenesis and a periodic supply of external nutrients}

\author{Huijuan Song, Qian Huang
and Zejia Wang\thanks{Corresponding author.
Email: zejiawang@jxnu.edu.cn}
\\
{\small \it School of Mathematics and Statistics, Jiangxi Normal University, }
\\
{\small \it Nanchang, Jiangxi 330022, P.R. China}
}

\date{}

\maketitle

\begin{abstract}
In this paper, we consider a free boundary problem modeling the growth of spherically symmetric necrotic tumors with angiogenesis and a $\omega$-periodic supply $\phi(t)$ of external nutrients. In the model, the consumption rate of the nutrient and the proliferation rate of tumor cells $S(\sigma)$ are both general nonlinear functions. The well-posedness and asymptotic behavior of solutions
are studied. We show that if the average of $S(\phi(t))$ is nonpositive, then all evolutionary tumors will finally vanish; the converse is also ture. If instead the average of $S(\phi(t))$ is positive, then there exists a unique positive periodic solution and all other evolutionary tumors will converge to this periodic state.
\\
{2020MSC:} 35R35, 35B35, 35Q92
\end{abstract}

\let\oldsection\section
\def\SEC{\oldsection}
\renewcommand\section{\setcounter{equation}{0}\SEC}
\renewcommand\thesection{\arabic{section}}

\renewcommand\theequation{\thesection.\arabic{equation}}
\newtheorem{proposition}{Proposition}[section]
\newtheorem{lemma}{Lemma}[section]
\newtheorem{theorem}{Theorem}[section]
\newtheorem{remark}{Remark}[section]
\newtheorem{corollary}{Corollary}[section]
\def\pd#1#2{\dfrac{\partial#1}{\partial#2}}
\allowdisplaybreaks
\renewcommand{\proofname}{\indent\it\bfseries Proof}

\section{Introduction}

It has been recognized that a tumor contains different populations of cells, such as proliferating cells, necrotic cells and ``in-between'' quiescent cells. During 1970's, Greenspan \cite{Gr-72,Gr-76} proposed the first mathematical model in the form of free boundary problem of reaction diffusion equations to explain this phenomenon. Since then a variety of PDE models for tumor growth have been developed, cf. \cite{B-X(13),B-F(05),B-C(95),B-C(96),Cui(06),F-MB(05),F-L(15)}, the review articles \cite{Low,N(05)} and the references cited therein. Besides rich numerical results, many rigorous mathematical analysis results including existence, uniqueness and stability theorems have been obtained, cf. \cite{B-X(13),B-F(05),B-E-Z(08),C-F(01),Cui(05),Cui(06),F-MB(05),F-L(15),H-X(21),H-Z-H(19),
S-H-W(21),SHW(21),W-W(19),W-X(20),X-S(20),Z-X(14),Z-C(18)} and the references given therein.

In this paper we study the following free boundary problem modeling the growth of radially symmetric necrotic tumors with angiogenesis and a periodic supply of external nutrients:
\begin{align}
    &\frac1{r^2}\frac{\partial}{\partial r}\left(r^2\frac{\partial\sigma}{\partial r}\right)= f(\sigma)H(\sigma-\sigma_D), &&0<r<R(t),~t>0,
    \label{eq(1.1)}
    \\
    &\frac{\partial\sigma}{\partial r}(0,t)=0,&&t>0,
    \label{eq(1.2)}
    \\
    &\frac{\partial\sigma}{\partial r}(R(t),t)+\beta[\sigma(R(t),t)-\phi(t)]=0,&&t>0,
    \label{eq(1.3)}
    \\
    &R^2(t)\frac{dR(t)}{dt}=\int^{R(t)}_0\big\{g(\sigma)H(\sigma-\sigma_D)-\nu[1-H(\sigma-\sigma_D)]
    \big\}r^2dr,&&t>0,
    \label{eq(1.4)}
    \\
    &R(0)=R_0.&&
    \label{eq(1.5)}
\end{align}
Here $\sigma(r,t)$ denotes the nutrient concentration within the tumor, $r=|x|$, $x\in\mathbb{R}^3$, $R(t)$ is the tumor radius, $H(s)$ is the Heaviside function: $H(s)=0$ for $s\leq0$ and $H(s)=1$ for $s>0$, the term $f(\sigma)H(\sigma-\sigma_D)$ on the right-hand side of \eqref{eq(1.1)} and
\begin{equation}
\label{eq(1.7)}
S(\sigma)=g(\sigma)H(\sigma-\sigma_D)-\nu[1-H(\sigma-\sigma_D)]
\end{equation}
represent the nutrient consumption rate and the tumor-cell proliferation rate functions, respectively, $\beta$, $\sigma_D$ and $\nu$ are positive constants, with $\beta$ depending on the density of the blood vessels, $\sigma_D$ a threshold value of nutrient concentration
to sustain tumor cells alive and proliferating, i.e., only in the region where $\sigma>\sigma_D$ tumor cells are alive and proliferating, and $\nu$ the dissolution rate of necrotic cells, $\phi(t)$ is a positive and continuous periodic function of period $\omega>0$, which reflects that the nutrient concentration in the host tissue is periodically changed, and $R_0>0$ is a given initial tumor radius.

Before going to our interest, we prefer to recall some relevant works. The model \eqref{eq(1.1)}--\eqref{eq(1.5)} with linear consumption and proliferation rates:
\begin{equation}
\label{eq(1.6)}
f(\sigma)=\sigma,\quad g(\sigma)=\mu(\sigma-\tilde\sigma)\quad(\mu>0,~\tilde\sigma>0)
\end{equation}
and the Dirichlet boundary condition
\begin{equation}
\label{eq(1.8)}
\sigma(R(t),t)=\bar\sigma\quad(\bar\sigma>0),
\end{equation}
was proposed by Cui \cite{Cui(06)} as in essence a combination of two
Byrne-Chaplain inhibitor-free and avascular tumor models; see \cite{B-C(95)} for the nonnecrotic case and then \cite{B-C(96)} for
its necrotic version. This is made such that both nonnecrotic tumors and necrotic tumors can be considered
in a joint way. By delicate calculations based on the existence of an explicit form for solutions of \eqref{eq(1.1)},
\eqref{eq(1.2)}, \eqref{eq(1.8)} for given $R(t)$, Cui \cite{Cui(06)}
studied the existence, uniqueness and global asymptotic stability of stationary solutions, the dependence on the parameters $\nu$ and $\bar\sigma$, as well as the mutual transition between the nonnecrotic and necrotic phases. When instead of \eqref{eq(1.8)}, the Robin boundary condition
\begin{equation}
\label{eq(1.9)}
\frac{\partial\sigma}{\partial r}(R(t),t)+\beta[\sigma(R(t),t)-\bar\sigma]=0
\end{equation}
is prescribed on the tumor boundary, the model was discussed by Xu-Su \cite{X-S(20)}. Recently, Wu-Wang \cite{W-W(19)} and Song-Hu-Wang \cite{SHW(21)} extended the results to the cases of general nonlinear functions $f$ and $g$, respectively. Since no explicit solution is available now, much more profound relations between all unknown functions were investigated.

In reality, many animals including humans have regular feeding activities, associated with the biological rhythm, and the nutrient concentration in their blood may change periodically over time, cf. \cite{Fo(17)}. Clearly, tumor models established under the assumption that the external nutrient concentration is a periodic function are more reasonable and consistent with real life.
In 2013, Bai-Xu \cite{B-X(13)} considered the nonnecrotic case, i.e., $\sigma>\sigma_D$ for $0<r<R(t)$, and imposed the following periodic boundary condition
\begin{equation}
\label{eq(1.10)}
\sigma(R(t),t)=\phi(t)
\end{equation}
with $\phi(t)$ given above. Assuming \eqref{eq(1.6)}, they obtained the necessary condition
\begin{equation*}
\frac1{\omega}\int^\omega_0\phi(t)dt-\tilde\sigma\le0
\end{equation*}
and sufficient condition
\begin{equation*}
\frac1{\omega}\int^\omega_0\phi(t)dt-\tilde\sigma<0
\end{equation*}
for the global stability of zero steady state, and if
\begin{equation*}
\min_{0\le t\le\omega}\phi(t)-\tilde\sigma>0,
\end{equation*}
then there exists a unique positive periodic solution, which is the global attractor of all other positive solutions. These results were very well improved by He-Xing \cite{H-X(21)} in 2021. Precisely speaking, they proved that zero steady state is still globally stable when $\frac1{\omega}\int^\omega_0\phi(t)dt-\tilde\sigma=0$, and a unique positive periodic solution exists if and only if $\frac1{\omega}\int^\omega_0\phi(t)dt-\tilde\sigma>0$, which is also globally stable under radial perturbations. Very recently, Wu-Xu \cite{W-X(20)} analyzed the model \eqref{eq(1.1)}--\eqref{eq(1.5)} with \eqref{eq(1.3)} replaced by the boundary condition \eqref{eq(1.10)}, and complete existence, uniqueness and stability results were given. Finally, for other related study, we refer the reader to \cite{B-F(05),B-E-Z(08),C-F(01),Cui(05),F-MB(05),H-Z-H(19),S-H-W(21),Hao,Z-X(14)} and the references cited therein.

Motivated by the above works, in this paper we are concerned with the model \eqref{eq(1.1)}--\eqref{eq(1.5)},
where in \eqref{eq(1.3)} the impact from angiogenesis is incorporated (see \cite{F-L(15)}). We aim at investigating the asymptotic behavior of all transient solutions as well as the effects of angiogenesis and a periodic supply of external nutrients on the dynamics of necrotic tumor growth. Throughout we shall assume
\\
(A1) $f\in C^1[0,+\infty)$ is strictly increasing and $f(0)=0$;
\\
(A2) $g \in C^1[0,+\infty)$ is strictly increasing and $g(\tilde\sigma)=0$ for some $\tilde\sigma>\sigma_D$;
\\
(A3) $g(\sigma_D)+\nu\geq0$.\\
Here the relation $g(\sigma_D)+\nu \geq 0$ means that the volume loss rate of living cells at $\sigma_D$ is not greater than the dissolution rate of necrotic cells; see \cite{Cui(06)} for a detailed derivation.

Denote
\begin{equation*}
\bar{S}=\frac1{\omega}\int^{\omega}_0S(\phi(t))dt,
\end{equation*}
which represents the average of the proliferation rate when taking up the nutrient in the host tissue. We are now ready to present the main result of this paper.

\begin{theorem}
\label{thm-1.1}
Suppose the assumptions (A1)--(A3) are satisfied. Then for any $R_0>0$, the problem \eqref{eq(1.1)}--\eqref{eq(1.5)} has a unique solution $(\sigma(r,t),R(t,R_0))$ $(0\le r\le R(t,R_0))$ for all $t\ge0$. Moreover,
\\
(i) if $\bar{S}\le0$, then $\lim_{t\to+\infty}R(t,R_0)=0$ for any given initial value $R_0>0$;
\\
(ii) if $\lim_{t\to+\infty}R(t,R_0)=0$ for some initial value $R_0>0$,
then $\bar{S}\le0$;
\\
(iii) if $\bar{S}>0$, then \eqref{eq(1.1)}--\eqref{eq(1.4)} admits a unique positive $\omega$-time periodic solution $(\sigma_{\rm per}(r,t),R_{\rm per}(t))$ $(0\le r\le R_{\rm per}(t))$ with the property that for any initial value $R_0>0$,
\begin{equation*}
\lim_{t\to+\infty}|R(t,R_0)-R_{\rm per}(t)|=0.
\end{equation*}
\end{theorem}

We give a complete classification of asymptotic behavior of solutions to the model \eqref{eq(1.1)}--\eqref{eq(1.5)} in Theorem \ref{thm-1.1}. It is easy to check that our results are consistent with those in \cite{H-X(21),W-X(20)}. In the model \eqref{eq(1.1)}--\eqref{eq(1.5)}, besides involving the general nonlinear nutrient consumption rate and the tumor-cell proliferation rate functions, the Robin boundary condition and a periodic supply function of external nutrients, there may exist a necrotic core, and so the tumor may have two different types of free boundaries. All these make the study much more difficult and challenging. The idea of our proof is inspired by that of \cite{H-X(21),W-X(20),Z-C(18)}, which relies essentially on the analysis made in our previous work \cite{SHW(21)} on the particular case $\phi(t)\equiv\bar\sigma$.

The rest of this paper is organized as follows. In Section 2, we collect and establish some auxiliary lemmas. In Section 3, we carry out the proof of Theorem \ref{thm-1.1}. In the last section, we give some conclusions and biological interpretations.

\section{Preliminaries}

In this section, we first collect some known results from \cite{SHW(21)} and then establish some new properties, which will be used in the next section.

Given $\bar\sigma>0$ and $R>0$, consider the problem
\begin{equation}
\label{eq-2.1}
\begin{cases}
    \sigma''(r)+\frac{2}{r}\sigma'(r)=f(\sigma)H(\sigma-\sigma_D),\quad0<r<R,
    \\
    \sigma'(0)=0,\quad\sigma'(R)+\beta(\sigma(R)-\bar\sigma)=0.
\end{cases}
\end{equation}
Then the next result follows immediately from \cite{SHW(21)}.

\begin{lemma}
\label{lem-2.1}
Let (A1) hold. Then the problem \eqref{eq-2.1} admits a unique solution $\sigma(r,\bar\sigma,R)$. Moreover,
the following assertions are true:\\
(a) If $0<\bar\sigma\le\sigma_D$, then for any $R>0$,
\begin{equation}
\label{eq-2.2}
\sigma(r,\bar\sigma,R)\equiv\bar\sigma,\quad0\le r\le R.
\end{equation}
(b) If $\bar\sigma>\sigma_D$, then there exists a threshold $R_c$ for the parameter $R$, which depends on $\bar\sigma$ and satisfies
$R_c'(\bar\sigma)>0$ for $\bar\sigma>\sigma_D$,
\begin{equation}
\label{eq-2.3}
\lim_{\bar\sigma\to\sigma_D^+}R_c(\bar\sigma)=0.
\end{equation}
Precisely speaking,

(b1) if $0<R\leq R_c(\bar\sigma)$, then $\sigma(0,\bar\sigma,R)\ge\sigma_D$, $\sigma_r(r,\bar\sigma,R)>0$ for $0<r\le R$ and $\sigma(r,\bar\sigma,R)$ solves
\begin{equation}
\label{eq-2.4}
\begin{cases}
    \frac{\partial^2\sigma}{\partial r^2}(r,\bar\sigma,R)
    +\frac{2}{r}\frac{\partial\sigma}{\partial r}(r,\bar\sigma,R)
    =f(\sigma(r,\bar\sigma,R)),\quad0<r<R,
    \\
   \frac{\partial\sigma}{\partial r}(0,\bar\sigma,R)=0,
   \quad\frac{\partial\sigma}{\partial r}(R,\bar\sigma,R)+\beta[\sigma(R,\bar\sigma,R)-\bar\sigma]=0;
\end{cases}
\end{equation}

(b2) if $R>R_c(\bar\sigma)$, then
\begin{equation}
    \label{eq-2.5}
 \sigma(r,\bar\sigma,R)=
  \begin{cases}
  V(r,\bar\sigma,R),&\rho(\bar\sigma,R)\le r\le R, \\
  \sigma_D,&0\le r<\rho(\bar\sigma,R),
  \end{cases}
\end{equation}
where $(V(r,\bar\sigma,R),\rho(\bar\sigma,R))$ solves
\begin{equation}
\label{eq-2.6}
\begin{cases}
    \frac{\partial^2V}{\partial r^2}(r,\bar\sigma,R)+\frac{2}{r}\frac{\partial V}{\partial r}(r,\bar\sigma,R)=f(V(r,\bar\sigma,R)),\quad\rho(\bar\sigma,R)<r<R,
    \\
   \frac{\partial V}{\partial r}(\rho(\bar\sigma,R),\bar\sigma,R)=0,
   \quad \frac{\partial V}{\partial r}(R,\bar\sigma,R)+\beta[V(R,\bar\sigma,R)-\bar\sigma]=0,
   \\
   V(\rho(\bar\sigma,R),\bar\sigma,R)=\sigma_D,
\end{cases}
\end{equation}
and $\frac{\partial V}{\partial r}(r,\bar\sigma,R)>0$ for $\rho(\bar\sigma,R)<r\le R$,
$\frac{\partial}{\partial R}\left(\rho(\bar\sigma,R)/R\right)>0$ for $R>R_c(\bar\sigma)$,
\begin{equation}
\label{eq-2.7}
\lim_{R\to R^+_c(\bar\sigma)}\rho(\bar\sigma,R)=0,
\quad\lim_{R\to+\infty}\frac{\rho(\bar\sigma,R)}R=1.
\end{equation}
\end{lemma}

\begin{remark}
\label{rem-2.1}
Fixing $\hat{R}>0$, one can obtain
\begin{equation}
\label{eq-2.8}
\lim_{(\bar\sigma,R)\to(\sigma_D,\hat{R})\atop \bar\sigma>\sigma_D}\rho(\bar\sigma,R)=\hat{R}.
\end{equation}
In fact, it follows from \eqref{eq-2.3} and \eqref{eq-2.6} that
$$
\frac1{R^2}\int^R_\rho l^2f(V)dl
+\beta\int^R_\rho\frac1{r^2}\int^r_\rho l^2f(V)dldr=\beta(\bar\sigma-\sigma_D).
$$
Thus,
\begin{equation*}
\beta(\bar\sigma-\sigma_D)\ge\frac1{R^2}\int^R_\rho l^2f(V)dl
\ge\frac{f(\sigma_D)}{R^2}\frac{R^3-\rho^3}3>0,
\end{equation*}
which proves \eqref{eq-2.8}.
\end{remark}

We define for any $\bar\sigma>0$ and any $R>0$,
\begin{equation}
\label{eq-2.9}
G(\bar\sigma,R)=\frac1{R^3}\left[\int_{\sigma(r,\bar\sigma,R)>\sigma_D}g(\sigma(r,\bar\sigma,R))r^2dr
-\int_{\sigma(r,\bar\sigma,R)\leq\sigma_D}\nu r^2dr\right],
\end{equation}
where $\sigma(r,\bar\sigma,R)$ is the solution of \eqref{eq-2.1}. Then Lemma \ref{lem-2.1}, combined with Remark \ref{rem-2.1} and the hypothesis (A2), shows that
\begin{equation}
\label{eq-2.10}
G(\bar\sigma,R)=
\begin{cases}
-\frac{\nu}3,&\quad 0<\bar\sigma\le\sigma_D,~R>0,
\\
\frac1{R^3}\int_0^R g(\sigma(r,\bar\sigma,R))r^2dr,&\quad\bar\sigma>\sigma_D,~0<R\le R_c(\bar\sigma),
\\
\frac1{R^3}\int_{\rho(\bar\sigma,R)}^R g(V(r,\bar\sigma,R))r^2dr-\frac{\nu}{3}\frac{\rho^3(\bar\sigma,R)}{R^3}, &\quad\bar\sigma>\sigma_D,~R>R_c(\bar\sigma),
\end{cases}
\end{equation}
and $G$ is continuous on $(0,+\infty)\times(0,+\infty)$. Furthermore, we have the following result.

\begin{lemma}
\label{lem-2.2}
Suppose that (A1), (A2) and (A3) are satisfied. Then
\\
(i) for $\bar\sigma>\sigma_D$ and $R>0$, $G(\bar\sigma,R)$ is strictly increasing in $\bar\sigma$ and strictly decreasing in $R$;
\\
(ii) for any fixed $\bar\sigma>\sigma_D$,
\begin{equation}
\label{eq-2.15}
    \lim_{R\to0^+}G(\bar\sigma,R)=\frac{g(\bar\sigma)}3,\quad
    \lim_{R\to+\infty}G(\bar\sigma,R)=-\frac{\nu}3.
\end{equation}
\end{lemma}

\begin{proof}
(i) For brevity, here we only prove the strictly increasing property of the function $G$ with respect to $\bar\sigma$, and refer the reader to \cite{SHW(21)} for the strictly decreasing property of $G$ with respect to $R$.
Given $R_*>0$, if there exists $\bar\sigma_*>\sigma_D$ such that $R_c(\bar\sigma_*)=R_*$, then we denote $\bar\sigma_*=R_c^{-1}(R_*)$; otherwise, $\bar\sigma_*=+\infty$. We claim that $\partial_{\bar\sigma}G(\bar\sigma,R_*)>0$ for $\sigma_D<\bar\sigma<\bar\sigma_*$ or $\bar\sigma>\bar\sigma_*$. In fact, when $\sigma_D<\bar\sigma<\bar\sigma_*$, one finds from \eqref{eq-2.10} and \eqref{eq-2.6} that
\begin{equation}
\label{eq-2.11}
\frac{\partial G}{\partial\bar\sigma}(\bar\sigma,R_*)=\frac1{R_*^3}\left\{\int^{R_*}_{\rho(\bar\sigma,R_*)}
g'(V(r,\bar\sigma,R_*))
\frac{\partial V}{\partial\bar\sigma}(r,\bar\sigma,R_*)r^2dr-[g(\sigma_D)+\nu]\rho^2(\bar\sigma,R_*)
\frac{\partial\rho}{\partial\bar\sigma}(\bar\sigma,R_*)\right\},
\end{equation}
and
\begin{equation}
\label{eq-2.12}
\begin{cases}
    -\frac{\partial^2}{\partial r^2}\left(\frac{\partial V}{\partial\bar\sigma}\right)(r,\bar\sigma,R_*)-\frac{2}{r}\frac{\partial}{\partial r}\left(\frac{\partial V}{\partial\bar\sigma}\right)(r,\bar\sigma,R_*)+f'(V(r,\bar\sigma,R_*))\frac{\partial V}{\partial\bar\sigma}(r,\bar\sigma,R_*)=0,\quad\rho(\bar\sigma,R_*)<r<R_*,
    \\
   \frac{\partial}{\partial r}\left(\frac{\partial V}{\partial\bar\sigma}\right)(\rho(\bar\sigma,R_*),\bar\sigma,R_*)=
-f(\sigma_D)\frac{\partial\rho}{\partial\bar\sigma}(\bar\sigma,R_*),
   \quad \frac{\partial}{\partial r}\left(\frac{\partial V}{\partial\bar\sigma}\right)(R_*,\bar\sigma,R_*)+\beta\frac{\partial V}{\partial\bar\sigma}(R_*,\bar\sigma,R_*)=\beta,
   \\
   \frac{\partial V}{\partial\bar\sigma}(\rho(\bar\sigma,R_*),\bar\sigma,R_*)=0.
\end{cases}
\end{equation}
Applying strong maximum principle to the problem \eqref{eq-2.12} yields
$$
\frac{\partial V}{\partial\bar\sigma}(r,\bar\sigma,R_*)>0\quad{\rm for ~each}~\rho(\bar\sigma,R_*)<r\le R_*\qquad{\rm and}~\qquad\frac{\partial\rho}{\partial\bar\sigma}(\bar\sigma,R_*)<0,
$$
which, combined with the assumptions (A2) and (A3), shows
\begin{equation}
\label{eq-2.13}
\frac{\partial G}{\partial\bar\sigma}(\bar\sigma,R_*)>0\quad{\rm for}~\sigma_D<\bar\sigma<\bar\sigma_*.
\end{equation}
While when $\bar\sigma>\bar\sigma_*$, there holds $R_*<R_c(\bar\sigma)$. Thus, by \eqref{eq-2.10} and \eqref{eq-2.4},
\begin{equation*}
\frac{\partial G}{\partial\bar\sigma}(\bar\sigma,R_*)
=\frac1{R_*^3}\int^{R_*}_0
g'(\sigma(r,\bar\sigma,R_*))
\frac{\partial\sigma}{\partial\bar\sigma}(r,\bar\sigma,R_*)r^2dr,
\end{equation*}
\begin{equation*}
\begin{cases}
    -\frac{\partial^2}{\partial r^2}\left(\frac{\partial \sigma}{\partial\bar\sigma}\right)(r,\bar\sigma,R_*)-\frac{2}{r}\frac{\partial}{\partial r}\left(\frac{\partial \sigma}{\partial\bar\sigma}\right)(r,\bar\sigma,R_*)+f'( \sigma(r,\bar\sigma,R_*))\frac{\partial \sigma}{\partial\bar\sigma}(r,\bar\sigma,R_*)=0,\quad0<r<R_*,
    \\
   \frac{\partial}{\partial r}\left(\frac{\partial \sigma}{\partial\bar\sigma}\right)(0,\bar\sigma,R_*)=0,
   \quad \frac{\partial}{\partial r}\left(\frac{\partial \sigma}{\partial\bar\sigma}\right)(R_*,\bar\sigma,R_*)+\beta\frac{\partial \sigma}{\partial\bar\sigma}(R_*,\bar\sigma,R_*)=\beta.
\end{cases}
\end{equation*}
Using strong maximum principle and the condition (A2) again, we obtain $\partial_{\bar\sigma}G(\bar\sigma,R_*)>0$ for $\bar\sigma>\bar\sigma_*$, which together with
\eqref{eq-2.13}, the continuity of the function $G$ and the arbitrariness of $R_*$ implies that $G(\bar\sigma,R)$ is strictly increasing in $\bar\sigma$ for $\bar\sigma>\sigma_D$.

(ii) For $\bar\sigma>\sigma_D$ and $R$ small enough,
\begin{equation}
\label{eq-2.14}
G(\bar\sigma,R)=\frac1{R^3}\int_0^R g(\sigma(r,\bar\sigma,R))r^2dr.
\end{equation}
If we set
$$
U(s,\bar\sigma,R)=\sigma(r,\bar\sigma,R)\quad{\rm with}~s=\frac{r}R,
$$
then \eqref{eq-2.14} becomes
$$
G(\bar\sigma,R)=\int^1_0g(U(s,\bar\sigma,R))s^2ds.
$$
Since $\lim_{R\to0^+}U(s,\bar\sigma,R)=\bar\sigma$ for every $0\le s\le1$ (see (2.18) in \cite{SHW(21)}), the first equality in \eqref{eq-2.15} immediately follows the Lebesgue dominant convergence theorem. On the other hand, for $\bar\sigma>\sigma_D$ and $R$ large enough, we see from
\eqref{eq-2.10} that
\begin{equation}
\label{eq-2.16}
G(\bar\sigma,R)=\frac1{R^3}\int_{\rho(\bar\sigma,R)}^R g(V(r,\bar\sigma,R))r^2dr-\frac{\nu}{3}\frac{\rho^3(\bar\sigma,R)}{R^3}.
\end{equation}
Based on the conclusion (b2) in Lemma \ref{lem-2.1}, sending $R\to+\infty$ in \eqref{eq-2.16} yields the second equality in \eqref{eq-2.15}. The proof is complete.
\end{proof}

\begin{remark}
\label{rem-2.2}
We stress that the assumptions (A2) and (A3) tell us that the function $S(\bar\sigma)$, given by  \eqref{eq(1.7)}, is nondecreasing in $(0,+\infty)$. Moreover, we find from the definition of $G$ and Lemma \ref{lem-2.2} that
\begin{equation}
\label{eq-2.17}
G(\bar\sigma,R)\le\frac{S(\bar\sigma)}3\quad{\rm for~all}~\bar\sigma>0~{\rm and}~R>0.
\end{equation}
\end{remark}

\section{Proof of the main result}

In this section, we will prove Theorem \ref{thm-1.1}. We first establish the existence and uniqueness of transient solutions to the system \eqref{eq(1.1)}--\eqref{eq(1.5)} for any $R_0>0$.

\begin{proposition}
\label{prop-3.1}
Let the assumptions (A1)--(A3) hold. Then for any $R_0>0$, the problem \eqref{eq(1.1)}--\eqref{eq(1.5)} allows a unique solution $(\sigma(r,t),R(t,R_0))$ $(0\le r\le R(t,R_0))$ for all $t\ge0$.
\end{proposition}

\begin{proof}
By Lemma \ref{lem-2.1}, at any $t>0$ and for any given $R(t)>0$, \eqref{eq(1.1)}--\eqref{eq(1.3)} allows a unique solution $\sigma(r,\phi(t),R(t))$. Substituting it into \eqref{eq(1.4)}, we obtain an initial value problem:
\begin{equation}
\label{eq-3.1}
\begin{cases}
\frac{dR}{dt}=R(t)G(\phi(t),R(t))\quad{\rm for}~t>0,\\
R(0)=R_0,
\end{cases}
\end{equation}
where the function $G$ is defined by \eqref{eq-2.9}.
Then based on Lemma \ref{lem-2.2}, the standard theory of ordinary differential equations shows that the problem \eqref{eq-3.1} has a unique global solution $R(t,R_0)$ and
\begin{equation}
\label{eq-3.6}
R(t,R_0)\ge R_0e^{-\nu t/3}
\end{equation}
for all $t\ge0$. We thus complete the the proof of this proposition.
\end{proof}

Before proceeding to the discussion of the stability of zero steady state of \eqref{eq-3.1}, we need some more notations. Set
\begin{equation}
\label{eq-3.2}
S^*=\max_{0\le t\le\omega}\int^t_0S(\phi(\tau))d\tau,
\quad S_*=\min_{0\le t\le\omega}\int^t_0S(\phi(\tau))d\tau.
\end{equation}

\begin{proposition}
\label{prop-3.2}
Assume the assumptions (A1)--(A3) are satisfied. Let $R(t,R_0)$ be the unique solution to the problem \eqref{eq-3.1}. Then the following conclusions are true:
\\
(i) If $\bar{S}\le0$, then $\lim_{t\to+\infty}R(t,R_0)=0$ for any given initial value $R_0>0$;
\\
(ii) If $\lim_{t\to+\infty}R(t,R_0)=0$ for some initial value $R_0>0$,
then $\bar{S}\le0$.
\end{proposition}

\begin{proof}
For convenience, the proof of the assertion (i) is split into three steps.

{\bf Step 1.} We show that for any $a\ge0$ and $t\ge a$,
\begin{equation}
\label{eq-3.0}
R(t,R_0)\le R(a,R_0)e^{\frac{2S^*-S_*}3}.
\end{equation}
Indeed, given $t_2\ge t_1\ge0$, it follows from \eqref{eq-3.1}, \eqref{eq-3.6} and \eqref{eq-2.17} that
\begin{equation}
\label{eq-3.3}
R(t_2,R_0)=R(t_1,R_0)e^{\int^{t_2}_{t_1}G(\phi(\tau),R(\tau,R_0))d\tau}
\le R(t_1,R_0)e^{\frac13\int^{t_2}_{t_1}S(\phi(\tau))d\tau}.
\end{equation}
As a straightforward consequence, we obtain
\begin{equation}
R(t+\omega,R_0)\le R(t,R_0)\quad{\rm for~all}~t\ge0.
\label{eq-3.4}
\end{equation}
On the other hand, letting $a\le t\le a+\omega$, we obtain from \eqref{eq-3.3} that
\begin{equation}
\label{eq-3.13}
R(t,R_0)\le R(a,R_0)e^{\frac13\int^{t}_aS(\phi(\tau))d\tau}.
\end{equation}
We further prove that
\begin{equation}
\label{eq-3.10}
\int^{t}_aS(\phi(\tau))d\tau\le2S^*-S_*\quad{\rm for}~a\le t\le a+\omega.
\end{equation}
Since there exists a unique nonnegative integer $k$ such that $k\omega\le a<(k+1)\omega$, there holds $k\omega\le a\le t<(k+1)\omega$ or $k\omega\le a<(k+1)\omega\le t\le a+\omega<(k+2)\omega$ for $a\le t\le a+\omega$. If $k\omega\le a\le t<(k+1)\omega$, then
\begin{equation*}
\int^t_aS(\phi(\tau))d\tau=\int^t_{k\omega}S(\phi(\tau))d\tau-\int^a_{k\omega}S(\phi(\tau))d\tau
\le S^*-S_*,
\end{equation*}
while if $k\omega\le a<(k+1)\omega\le t\le a+\omega<(k+2)\omega$, then
\begin{equation*}
\begin{split}
\int^t_aS(\phi(\tau))d\tau&=\int^{(k+1)\omega}_aS(\phi(\tau))d\tau
+\int^t_{(k+1)\omega}S(\phi(\tau))d\tau
\\
&=\int^{(k+1)\omega}_{k\omega}S(\phi(\tau))d\tau-\int^a_{k\omega}S(\phi(\tau))d\tau
+\int^t_{(k+1)\omega}S(\phi(\tau))d\tau
\\
&\le 2S^*-S_*,
\end{split}
\end{equation*}
which proves \eqref{eq-3.10}. Combining \eqref{eq-3.13} and \eqref{eq-3.10}, we derive
\begin{equation}
R(t,R_0)\le R(a,R_0)e^{\frac{2S^*-S_*}3}\quad{\rm for~any}~a\le t\le a+\omega.
\label{eq-3.5}
\end{equation}
Finally, \eqref{eq-3.0} follows from \eqref{eq-3.4} and \eqref{eq-3.5}.

{\bf Step 2.} We claim that
\begin{equation}
\label{eq-3.7}
\liminf_{t\to+\infty}R(t,R_0)=0.
\end{equation}
Argue by contradiction. Suppose that \eqref{eq-3.7} is false and then in view of \eqref{eq-3.4},
\begin{equation}
\label{eq-3.8}
\liminf_{t\to+\infty}R(t,R_0)=\alpha
\end{equation}
for some $\alpha>0$. As a result, there exists $M>0$ such that
\begin{equation}
\label{eq-3.11}
R(t,R_0)\ge\frac{\alpha}2\quad{\rm for~each}~t\ge M.
\end{equation}
Hence, for any fixed $t_*>M$ and any positive integer $n$, one can apply Lemma \ref{lem-2.2} to derive
\begin{equation}
\label{eq-3.9}
R(t_*+n\omega,R_0)\le R(t_*,R_0)e^{\int^{t_*+n\omega}_{t_*}G\left(\phi(\tau),\frac\alpha2\right)d\tau}
=R(t_*,R_0)e^{n\int^{\omega}_0G\left(\phi(\tau),\frac\alpha2\right)d\tau}.
\end{equation}
Besides, the condition $\bar{S}\le0$, combined with (A2), (A3) and Lemma \ref{lem-2.2}, implies
\begin{equation*}
\int^{\omega}_0G\left(\phi(\tau),\frac\alpha2\right)d\tau<0.
\end{equation*}
Letting $n\to\infty$ in \eqref{eq-3.9}, we therefore arrive at
$$
\lim_{n\to\infty}R(t_*+n\omega,R_0)=0,
$$
which contradicts \eqref{eq-3.11}. This proves our claim.

{\bf Step 3.} We now prove that
\begin{equation}
\label{eq-3.12}
\lim_{t\to+\infty}R(t,R_0)=0.
\end{equation}
For any given $\varepsilon>0$, we see from \eqref{eq-3.7} that there exists $t_\varepsilon>0$ such that
$$
0<R(t_\varepsilon,R_0)<\varepsilon e^{-\frac{2S^*-S_*}3},
$$
which together with \eqref{eq-3.0} implies
$$
R(t,R_0)<\varepsilon\quad{\rm for~any}~t\ge t_\varepsilon;
$$
thus, \eqref{eq-3.12} holds true.

Next, we show that if $\lim_{t\to+\infty}R(t,R_0)=0$ for some $R_0>0$, then $\bar{S}\le0$. If not, there holds $\bar{S}>0$. Combining \eqref{eq-2.15} and the Lebesgue dominated convergence theorem, we get
$$
\lim_{R\to0^+}\int^{\omega}_0 G(\phi(t),R)dt=\frac13\int^{\omega}_0 S(\phi(t))dt=\frac{\omega\bar{S}}3>0.
$$
Thus, there exists $\delta>0$ such that $\int^{\omega}_0 G(\phi(t),\delta)dt>0$. On the other hand, for the positive constant $\delta$ above, $\lim_{t\to+\infty}R(t,R_0)=0$ implies that there exists $T>0$ such that
$$
R(t,R_0)\le\delta\quad{\rm for ~any}~t\ge T.
$$
Consequently, using the assertion (i) of Lemma \ref{lem-2.2} we arrive at for each $t\ge T$,
\begin{equation*}
\begin{split}
R(t+\omega,R_0)&=R(t,R_0)e^{\int^{t+\omega}_t G(\phi(\tau),R(\tau,R_0))d\tau}
\ge R(t,R_0)e^{\int^{t+\omega}_t G(\phi(\tau),\delta)d\tau}
\\
&=R(t,R_0)e^{\int^{\omega}_0 G(\phi(\tau),\delta)d\tau}
>R(t,R_0),
\end{split}
\end{equation*}
which yields a contradiction. We therefore obtain $\bar{S}\le0$ and the proof is complete.
\end{proof}

\begin{remark}
\label{rem-3.1}
We stress that in obtaining the assertion (i) of Proposition \ref{prop-3.2}, the difficult part is the proof of the global stability of zero steady state when $\bar{S}=0$. In fact, if $\bar{S}<0$, we see
from \eqref{eq-3.3} that for any fixed $t_0\ge0$,
\begin{equation}
\label{eq-3.24}
R(t_0+n\omega,R_0)\le R(t_0,R_0)e^{\frac13\int^{t_0+n\omega}_{t_0}S(\phi(\tau))d\tau}
=R(t_0,R_0)e^{\frac13n\omega\bar{S}},
\end{equation}
which implies
$$
\lim_{n\to\infty}R(t_0+n\omega,R_0)=0.
$$
Furthermore, one can easily establish $\lim_{t\to+\infty}R(t,R_0)=0$. However, if $\bar{S}=0$, \eqref{eq-3.24} reduces to
$$
R(t_0+n\omega,R_0)\le R(t_0,R_0),
$$
which fails to yield the desired result. For this reason, we here have employed the method used in
\cite{H-X(21)}. Concretely speaking, we first apply the properties that $\{R(t+n\omega,R_0)\}$ is a nonincreasing sequence in $n$ (see \eqref{eq-3.4}) and in one period, $R(t,R_0)$, $t\in[a,a+\omega]$, can be controlled by $CR(a,R_0)$ (see \eqref{eq-3.5}) to arrive at \eqref{eq-3.0}. Then combing with \eqref{eq-3.7}, we conclude the assertion.
\end{remark}

Finally, we discuss the existence, uniqueness and stability of positive periodic solutions to \eqref{eq-3.1}.

\begin{proposition}
\label{prop-3.3}
Suppose that (A1)--(A3) are fulfilled. Let $\bar{S}>0$. Then the problem
\eqref{eq-3.1} admits a unique positive $\omega$-periodic solution $R_{\rm per}(t)$. Moreover, for any initial value $R_0>0$, the solution $R(t,R_0)$ to \eqref{eq-3.1} satisfies
\begin{equation}
\label{eq-3.14}
\lim_{t\to+\infty}|R(t,R_0)-R_{\rm per}(t)|=0.
\end{equation}
\end{proposition}

\begin{proof}
We first establish the existence of positive $\omega$-periodic solutions. Let us denote $\phi^*=\max_{0\le t\le\omega}\phi(t)$. Then $\bar{S}>0$ implies $\phi^*>\tilde\sigma>\sigma_D$ and $g(\phi^*)>0$. Thus, we obtain from Lemma \ref{lem-2.2} that there exists a unique $R^+>0$
such that $G(\phi^*,R^+)=0$; moreover, for any $R_0>0$ and any $t\ge0$,
\begin{equation}
\label{eq-3.15}
R(t,R_0)\le\max\{R_0,R^+\}.
\end{equation}
Given $R_0$, set $R_n=R(n\omega,R_0)$, $n=1,2,\cdots$. By the uniqueness of solutions to the problem \eqref{eq-3.1}, we drive $R_n=R(\omega,R_{n-1})$ and $\{R_n\}_{n\ge0}$ is a monotone sequence. Based on \eqref{eq-3.15}, the use of the convergence of bounded monotone sequences gives that $\{R_n\}$ converges to some nonnegative constant $R^\#$ as $n\to\infty$. We claim that
\begin{equation}
\label{eq-3.16}
R^\#>0.
\end{equation}
As a matter of fact, it suffices to consider the possibility that $R_0>R_1> R_2>\cdots>R_n>R_{n+1}>\cdots$. In this case, an argument similar to that for \eqref{eq-3.0} says that
\begin{equation*}
R(t,R_0)\le R_ne^{\frac13 S^*}\quad{\rm for~all}~n\ge0~{\rm and~all}~t\ge n\omega.
\end{equation*}
Hence, if $\lim_{n\to\infty}R_n=R^\#=0$, then $\lim_{t\to+\infty}R(t,R_0)=0$, which leads to $\bar{S}\le0$ by the second assertion of Proposition \ref{prop-3.2};
it contradicts the fact that $\bar{S}>0$ and the desired result \eqref{eq-3.16} follows. We now show that the solution $R(t,R^\#)$ is an $\omega$-periodic function. Since
$$
R_{n+1}=R_n e^{\int^\omega_0G(\phi(\tau),R(\tau,R_n))d\tau},
$$
we have
$$
\lim_{n\to\infty}\int^\omega_0G(\phi(\tau),R(\tau,R_n))d\tau=0,
$$
or equivalently,
\begin{equation}
\label{eq-3.17}
\int^\omega_0G(\phi(\tau),R(\tau,R^\#))d\tau=0.
\end{equation}
We then get from \eqref{eq-3.17} that
$$
R(\omega,R^\#)=R^\# e^{\int^\omega_0G(\phi(\tau),R(\tau,R^\#))d\tau}=R^\#,
$$
which together with the uniqueness of solutions of \eqref{eq-3.1} produces that $R(t+\omega,R^\#)=R(t,R^\#)$
for any $t\ge0$, as desired.

Next, we assume that $R_{\rm per}(t)$ is a positive $\omega$-periodic solution of \eqref{eq-3.1}, and prove that \eqref{eq-3.14} holds for all $R_0>0$. Without loss of generality, we may suppose that $R(t,R_0)<R_{\rm per}(t)$ for any $t\ge0$. If we write
\begin{equation}
\label{eq-3.18}
R(t,R_0)=R_{\rm per}(t)e^{y(t)},
\end{equation}
then $y(t)<0$ on $[0,+\infty)$ and
\begin{equation}
\label{eq-3.19}
\frac{dy}{dt}=G(\phi(t),R_{\rm per}(t)e^{y(t)})-G(\phi(t),R_{\rm per}(t))\quad{\rm for}~t>0
\end{equation}
by \eqref{eq-3.1}. Using Lemma \ref{lem-2.2}, we find $dy/dt\ge0$. Furthermore, we show that
\begin{equation}
\label{eq-3.20}
\lim_{t\to+\infty}y(t)=y_{\infty}=0.
\end{equation}
If \eqref{eq-3.20} is not true, then $y_{\infty}<0$
and $y(t)\le y_\infty$ for $t\ge0$. Thus, it follows from \eqref{eq-3.19} that
\begin{equation}
\label{eq-3.21}
\frac{dy}{dt}\ge G(\phi(t),R_{\rm per}(t)e^{y_\infty})-G(\phi(t),R_{\rm per}(t))\ge0\quad{\rm for}~t>0.
\end{equation}
On the other hand, notice that $\bar{S}>0$ implies that there exists some closed interval $[a,b]\subset[0,\omega]$
such that $\phi(t)>\sigma_D$ on $[a,b]$. Consequently, we employ Lemma \ref{lem-2.2} again to get
\begin{equation}
\label{eq-3.22}
G(\phi(t),R_{\rm per}(t)e^{y_\infty})-G(\phi(t),R_{\rm per}(t))>0\quad{\rm for~each}~t\in[a,b].
\end{equation}
A combination of \eqref{eq-3.21} and \eqref{eq-3.22} yields
\begin{equation}
\label{eq-3.23}
\begin{split}
y(b+k\omega)-y(a)=&\int^{b+k\omega}_a\frac{dy}{dt}dt
\ge\int^{b+k\omega}_a\left[G(\phi(t),R_{\rm per}(t)e^{y_\infty})-G(\phi(t),R_{\rm per}(t))\right]dt
\\
\ge&\int^b_a\left[G(\phi(t),R_{\rm per}(t)e^{y_\infty})-G(\phi(t),R_{\rm per}(t))\right]dt
\\
&+\int^{b+\omega}_{a+\omega}\left[G(\phi(t),R_{\rm per}(t)e^{y_\infty})-G(\phi(t),R_{\rm per}(t))\right]dt
\\
&+\cdots+\int^{b+k\omega}_{a+k\omega}\left[G(\phi(t),R_{\rm per}(t)e^{y_\infty})-G(\phi(t),R_{\rm per}(t))\right]dt
\\
=&(k+1)\int^b_a\left[G(\phi(t),R_{\rm per}(t)e^{y_\infty})-G(\phi(t),R_{\rm per}(t))\right]dt
\\
\to&+\infty(k\to\infty),
\end{split}
\end{equation}
which contradicts the fact that $\lim_{t\to+\infty}y(t)$ exists.
We therefore conclude \eqref{eq-3.20}, which, combined with \eqref{eq-3.18}, immediately leads to \eqref{eq-3.14}.

Finally, we obtain the uniqueness of positive $\omega$-periodic solutions to \eqref{eq-3.1}. Assume that $R^1_{\rm per}(t)$ and $R^2_{\rm per}(t)$ are two positive $\omega$-periodic solutions. Then we see from \eqref{eq-3.14} that for $t\ge0$,
$$
|R^1_{\rm per}(t)-R^2_{\rm per}(t)|=|R^1_{\rm per}(t+n\omega)-R^2_{\rm per}(t+n\omega)|\to0(n\to\infty),
$$
which implies $R^1_{\rm per}(t)=R^2_{\rm per}(t)$. This completes the proof.
\end{proof}

\begin{proof}[\indent\it\bfseries Proof of Theorem \ref{thm-1.1}]
The proof of Theorem \ref{thm-1.1} is now just a combination of that of Propositions \ref{prop-3.1}, \ref{prop-3.2} and
\ref{prop-3.3}.
\end{proof}

\section{Conclusion and biological interpretation}

In this paper, we have studied the well-posedness and asymptotic behavior of solutions to a free boundary problem modeling the growth of radially symmetric necrotic tumors with angiogenesis and a periodic supply of external nutrients. By denoting $\bar{S}=\frac1{\omega}\int^{\omega}_0S(\phi(t))dt$, representing the average of the proliferation rate when taking up the nutrient in the host tissue, we find that the external nutrient supply plays a critical role in tumor growth. 
Our results indicate that if $\bar{S}\le0$, being similar to the case of constant nutrient supply \cite{SHW(21)}, all evolutionary tumors would finally disappear. If $\bar{S}>0$, as a result of a periodic supply of external nutrients, the tumor would be induced a periodic growth state, and tumor growth would finally synchronize the periodic external nutrient supply. In other words, the growth pattern of tumors may be controlled by adjusting the external nutrient supply. These results may provide a reference for a better and deeper understanding of the tumor growth and the medical treatment.

\section*{Acknowledgments}
This work was partly supported by the National Natural Science Foundation of China (No. 11601200, No. 11861038, No. 12161045 and No. 11771156).


\begin{thebibliography}{[99]}
\bibitem{B-X(13)}
M. Bai and S. Xu, Qualitative analysis of a mathematical model for tumor growth with a periodic supply of external nutrients, {\it Pac. J. Appl. Math.}, {\bf5}(2013), 217--223.

\bibitem{B-F(05)}
M. Bodnar and U. Fory\'s, Time delay in necrotic core formation, {\it Math. Biosci. Eng.}, {\bf2}(2005), 461--472.

\bibitem{B-E-Z(08)}
H. Bueno, G. Ercole and A. Zumpano, Stationary solutions of a model for the growth of tumors and a connection between the nonnecrotic and necrotic phases,
{\it SIAM J. Appl. Math.}, {\bf68}(2008), 1004--1025.

\bibitem{B-C(95)}
H.M. Byrne and M.A.J. Chaplain, Growth of nonnecrotic tumors in the presence and absence of inhibitors, {\it Math. Biosci.}, {\bf130}(1995), 151--181.

\bibitem{B-C(96)}
H.M. Byrne and M.A.J. Chaplain, Growth of necrotic tumors in the presence and absence of inhibitors, {\it Math. Biosci.}, {\bf135}(1996), 187--216.

\bibitem{C-F(01)}
S. Cui and A. Friedman, Analysis of a mathematical model of the growth of necrotic tumors, {\it J.
Math. Anal. Appl.}, {\bf255}(2001), 636--677.

\bibitem{Cui(05)}
S. Cui, Global existence of solutions for a free boundary problem modeling the growth of necrotic tumors,
{\it Interfaces Free Bound.}, {\bf7}(2005), 147--159.

\bibitem{Cui(06)}
S. Cui, Formation of necrotic cores in the growth of tumors: analytic results, {\it Acta Math. Sci. Ser. B (Engl. Ed.)}, {\bf26}(2006), 781--796.

\bibitem{Fo(17)}
D.B. Forger, {\it Biological Clocks, Rhythms, and Oscillations: The Theory of Biological Timekeeping}, MIT Press, Cambridge, MA, 2017.

\bibitem{F-MB(05)}
U. Fory\'s and A. Mokwa-Borkowska, Solid tumour growth analysis of necrotic core formation, {\it Math. Comput. Modelling}, {\bf42}(2005), 593--600.

\bibitem{F-L(15)}
A. Friedman and K.-Y. Lam, Analysis of a free-boundary tumor model with angiogenesis, {\it J. Differential Equations}, {\bf259}(2015), 7636--7661.

\bibitem{Gr-72}
H.P. Greenspan, Models for the growth of a solid tumor by diffusion, {\it Stud. Appl. Math.}, {\bf51}(1972), 317--340.

\bibitem{Gr-76}
H.P. Greenspan, On the growth and stability of cell cultures and solid tumors, {\it J. Theor. Biol.}, {\bf56}(1976), 229--242.

\bibitem{Hao}
W. Hao, J.D. Hauenstein, B. Hu, Y. Liu, A.J. Sommese and Y. Zhang,
Bifurcation for a free boundary problem modeling the growth of a tumor with a necrotic core,
{\it Nonlinear Anal. Real World Appl.}, {\bf13}(2012), 694--709.

\bibitem{H-X(21)}
W. He and R. Xing, The existence and linear stability of periodic solution for a free boundary
problem modeling tumor growth with a periodic supply of external nutrients, {\it Nonlinear Anal. Real
World Appl.}, {\bf60}(2021), 103290, 1--16.

\bibitem{H-Z-H(19)}
Y. Huang, Z. Zhang and B. Hu, Linear stability for a free boundary tumor model with a
periodic supply of external nutrients, {\it Math. Methods Appl. Sci.}, {\bf42}(2019), 1039--1054.

\bibitem{Low}
J.S. Lowengrub, H.B. Frieboes, F. Jin, Y-L. Chuang, X. Li, P. Macklin, S.M. Wise and V. Cristini, Nonlinear modelling of cancer: bridging the gap between cells and tumours, {\it Nonlinearity}, {\bf23}(2010), 1--91.

\bibitem{N(05)}
J.D. Nagy, The ecology and evolutionary biology of cancer: A review of mathematical models
of necrosis and tumor cell diversity, {\it Math. Biosci. Eng.}, {\bf2}(2005), 381--418.

\bibitem{S-H-W(21)}
H. Song, B. Hu and Z. Wang, Stationary solutions of a free boundary problem modeling the
growth of vascular tumors with a necrotic core, {\it Discrete Contin. Dyn. Syst. Ser. B}, {\bf26}(2021),
667--691.

\bibitem{SHW(21)}
H. Song, W. Hu and Z. Wang, Analysis of a nonlinear free-boundary tumor model with angiogenesis and a connection between the nonnecrotic and necrotic phases, {\it Nonlinear Anal. Real World Appl.}, {\bf59}(2021), 103270, 1--11.

\bibitem{W-W(19)}
J. Wu and C. Wang, Radially symmetric growth of necrotic tumors and connection with nonnecrotic tumors, {\it Nonlinear Anal. Real World Appl.},
{\bf50}(2019), 25--33.

\bibitem{W-X(20)}
J. Wu and S. Xu, Asymptotic behavior of a nonlinear necrotic tumor model with a periodic external nutrient supply, {\it Discrete  Contin. Dyn. Syst. Ser. B},
{\bf25}(2020), 2453--2460.

\bibitem{X-S(20)}
S. Xu and D. Su, Analysis of necrotic core formation in angiogenic tumor growth, {\it Nonlinear Anal. Real World Appl.}, {\bf51}(2020), 103016, 1--12.

\bibitem{Z-X(14)}
F. Zhang and S. Xu, Steady-state analysis of necrotic core formation for solid
avascular tumors with time delays in regulatory apoptosis, {\it Comput. Math. Methods Med.}, {\bf 2014}(2014), 1--4.

\bibitem{Z-C(18)}
Y. Zhuang and S. Cui, Analysis of a free boundary problem modeling the growth of multicell spheroids with angiogenesis, {\it J. Differential Equations}, {\bf265}(2018), 620--644.
\end{thebibliography}
\end{document}